\documentclass[12pt,reqno]{article}

\usepackage[usenames]{color}
\usepackage[colorlinks=true,
linkcolor=webgreen, filecolor=webbrown,
citecolor=webgreen]{hyperref}

\definecolor{webgreen}{rgb}{0,.5,0}
\definecolor{webbrown}{rgb}{.6,0,0}

\usepackage{amssymb}
\usepackage{graphicx}
\usepackage{amscd}
\usepackage{lscape}
\usepackage{tikz}
\usepackage{tikz-cd}
\usepackage{pgfplots}

\usetikzlibrary{matrix}
\usetikzlibrary{fit,shapes}
\usetikzlibrary{positioning, calc}
\tikzset{circle node/.style = {circle,inner sep=1pt,draw, fill=white},
        X node/.style = {fill=white, inner sep=1pt},
        dot node/.style = {circle, draw, inner sep=5pt}
        }
\usepackage{tkz-fct}

\usepackage{amsthm}
\newtheorem{theorem}{Theorem}

\newtheorem{proposition}[theorem]{Proposition}

\newtheorem{conjecture}[theorem]{Conjecture}

\theoremstyle{definition}

\newtheorem{example}[theorem]{Example}

\usepackage{float}

\usepackage{graphics,amsmath}
\usepackage{amsfonts}
\usepackage{latexsym}
\usepackage{epsf}

\newcommand{\seqnum}[1]{\href{http://oeis.org/#1}{\underline{#1}}}

\begin{document}

\begin{center}
\vskip 1cm{\LARGE\bf Riordan arrays, the $A$-matrix, and Somos $4$ sequences} \vskip 1cm \large
Paul Barry\\
School of Science\\
Waterford Institute of Technology\\
Ireland\\
\href{mailto:pbarry@wit.ie}{\tt pbarry@wit.ie}
\end{center}
\vskip .2 in

\begin{abstract} We characterize certain Riordan arrays by their $A$-matrices and $\rho$ sequences. We conjecture the form of a generic $A$-matrix which leads to Somos $4$ sequences. We find an $A$-matrix that produces a Riordan quasi-involution, and we study the $A$-matrices and $\rho$ sequences of the moment matrices of certain perturbed orthogonal polynomials. \end{abstract}

\section{Introduction}

An important feature of Riordan arrays \cite{Book, Survey, SGWW} is that they have a sequence characterization. Specifically, a lower-triangular array $(t_{n,k})_{0 \le n,k \le \infty}$ is a Riordan array if and only if there exists a sequence $a_n, n\ge 0$, such that
$$t_{n,k}=\sum_{i=0}^{\infty}a_i t_{n-1,k-1+i}.$$
This sum is actually a finite sum, since the matrix is lower-triangular. For such a matrix $M$, we let  $$P=M^{-1}\overline{M},$$ where $\overline{M}$ is the matrix $M$ with its top row removed. Then the matrix $P$ has the form that begins
$$\left(
\begin{array}{ccccccc}
 z_0 & a_0 & 0 & 0 & 0 & 0 & 0 \\
 z_1 & a_1 & a_0 & 0 & 0 & 0 & 0 \\
 z_2 & a_2 & a_1 & a_0 & 0 & 0 & 0 \\
 z_3 & a_3 & a_2 & a_1 & a_0 & 0 & 0 \\
 z_4 & a_4 & a_3 & a_2 & a_1 & a_0 & 0 \\
 z_5 & a_5 & a_4 & a_3 & a_2 & a_1 & a_0 \\
 z_6 & a_6 & a_5 & a_4 & a_3 & a_2 & a_1 \\
\end{array}
\right).$$

Here, $z_0,z_1,z_2,\ldots$ is an ancillary sequence which exists for any lower-triangular matrix. The matrix $P$ is called the \emph{production matrix} of the Riordan matrix $M$. The sequence characterization of renewal arrays was first described by Rogers \cite{He_A, Rogers}.

An alternative approach to the sequence characterization of a Riordan array is to use a matrix characterization \cite{He_M, Merlini}.  One form that such a matrix characterization may take is the following \cite{He_M, Merlini}.
\begin{theorem} A lower-triangular array $(t_{n,k})_{0 \le n,k \le \infty}$ is a Riordan array if and only if there exists another array $A=(a_{i,j})_{i,j \in \mathbb{N}_0}$ with $a_{0,0} \ne 0$, and a sequence $(\rho_j)_{j \in \mathbb{N}_0}$ such that
$$t_{n+1,k+1}=\sum_{i \ge 0} \sum_{j \ge 0} a_{i,j} t_{n-i, k+j} + \sum_{j \ge 0} \rho_j t_{n+1,k+j+2}.$$
\end{theorem}
The power series definition of a Riordan array is as follows. A Riordan array is defined by a pair of power series, $g(x)$ and $f(x)$, where
$$g(x)=g_0 + g_1 x + g_2 x^2+ \cdots, \quad g_0 \ne 0,$$ and
$$f(x)=f_1 x + f_2 x^2+ f_3 x^3+\cdots, \quad f_0=0 \text{ and } f_1 \ne 0.$$
We then have
$$t_{n,k}=[x^n] g(x)f(x)^k,$$ where $[x^n]$ is the functional that extracts the coefficient of $x^n$.
The relationship between $f(x)$ and the pair $(A, \rho)$ is the following.
$$\frac{f(x)}{x}=\sum_{i \ge 0} x^i R^{(i)}(f(x))+\frac{f(x)^2}{x}\rho(f(x)),$$ where
$R^{(i)}$ is the generating series of the $i$-th row of $A$, and $\rho(x)$ is the generating series of the sequence $\rho_n$.

We shall call the matrix $A=(a_{i,j})_{i,j \in \mathbb{N}_0}$ the $A$-matrix, while we call the sequence $a_n$ the $A$-sequence of the Riordan array. The generating function of $a_n$ is denoted by $A(x)$. We have
$$A(x)=\frac{x}{\bar{f}(x)},$$ where $\bar{f}(x)$ is the compositional inverse of $f(x)$ (thus we have $f(\bar{f}(x))=x$ and $\bar{f}(f(x))=x$). We sometimes write $\bar{f}(x)=\text{Rev}(f)(x)$.

The generating function of the sequence $z_n$ is denoted by $Z(x)$.
We have
$$(g(x), f(x))=\left(\frac{A(x)-xZ(x)}{A(x)}, \frac{x}{A(x)}\right)^{-1}.$$
Note that while the production matrix $P$ of a Riordan array is unique, the pair $(A, \rho)$ is not necessarily unique.

The set of Riordan arrays $(g(x), f(x))$ is a group under the multiplication law
$$(g(x), f(x)) \cdot (u(x), v(x)) = (g(x)u(f(x)), v(f(x))$$ and the inverse of $(g(x), f(x))$ in the group is given by
$$(g(x), f(x))^{-1}=\left(\frac{1}{g(\bar{f}(x))}, \bar{f}(x)\right)$$ where
$\bar{f}(x)$ is the compositional inverse of $f(x)$.

The set of Riordan arrays of the form $\left(\frac{f(x)}{x}, f(x)\right)$ is a subgroup of the Riordan group, called the subgroup of Bell matrices.

We shall refer to sequences, where known, by their A$nnnnnn$ numbers in the Online Encyclopedia of Integer Sequences \cite{SL1, SL2}. The sequence $0^n$ is the sequence that begins $1,0,0,0,\ldots$, with generating function $x^0=1$. Note that we only show suitable truncations of Riordan arrays and of production arrays.
\begin{example} Pascal's triangle (also called the binomial matrix) $\left(\binom{n}{k}\right)$ is the Riordan array
$$\left(\frac{1}{1-x}, \frac{x}{1-x}\right).$$
This begins
$$\left(
\begin{array}{ccccccc}
 1 & 0 & 0 & 0 & 0 & 0 & 0 \\
 1 & 1 & 0 & 0 & 0 & 0 & 0 \\
 1 & 2 & 1 & 0 & 0 & 0 & 0 \\
 1 & 3 & 3 & 1 & 0 & 0 & 0 \\
 1 & 4 & 6 & 4 & 1 & 0 & 0 \\
 1 & 5 & 10 & 10 & 5 & 1 & 0 \\
 1 & 6 & 15 & 20 & 15 & 6 & 1 \\
\end{array}
\right).$$
The $A$-sequence for this triangle is given by $A(x)=1+x$. It can also be defined by the $A$-matrix
$$A=\left(
\begin{array}{ccc}
 1 & 0 & 0 \\
 0 & -1 & -1 \\
\end{array}
\right),$$ and the $\rho$ sequence $1,0,0,\ldots$.

This follow since the solution to the equation
$$\frac{u}{x}=1-x(u+u^2)+\frac{u^2}{x}$$ is
$$u=f(x)=\frac{x}{1-x}.$$
\end{example} 
We shall call the result of multiplying a Riordan array (on the left) by the binomial matrix as its \emph{binomial transform}. 
\begin{example} The identity matrix $(1,x)$ can be defined by the $A$-matrix
$$A=\left(
\begin{array}{ccc}
 1 & 0 & 1 \\
 -1 & -1 & 0 \\
\end{array}
\right),$$ and the $\rho$ sequence $1,0,0,\ldots$.
This follows since the equation
$$\frac{u}{x}=1+u^2-x(1+u)+\frac{u^2}{x}$$ has the solution $u=f(x)=x$.

Another $A$-matrix and $\rho$ sequence pair that define the identity matrix is given by
$$A=\left(
\begin{array}{ccc}
 1 & -1 & 1 \\
 0 & -1 & 0 \\
\end{array}
\right),$$ and the $\rho$ sequence $1,0,0,\ldots$.

It can also be defined by the simpler $A$-matrix $A=(1)$ and $\rho_n=0$.

\end{example}
\begin{example} We consider the Motzkin triangle defined by the Riordan array
$$\left(\frac{1-x-\sqrt{1-2x-3x^2}}{2x^2},\frac{1-x-\sqrt{1-2x-3x^2}}{2x}\right).$$ This can be expressed as
$$\left(\frac{1}{1-x}c\left(\frac{x^2}{(1-x)^2}\right), \frac{x}{1-x}c\left(\frac{x^2}{(1-x)^2}\right)\right),$$ where $c(x)=\frac{1-\sqrt{1-4x}}{2x}$ is the generating function of the Catalan numbers $C_n=\frac{1}{n+1} \binom{2n}{n}$ \seqnum{A000108}.

This array begins
$$\left(
\begin{array}{ccccccc}
 1 & 0 & 0 & 0 & 0 & 0 & 0 \\
 1 & 1 & 0 & 0 & 0 & 0 & 0 \\
 2 & 2 & 1 & 0 & 0 & 0 & 0 \\
 4 & 5 & 3 & 1 & 0 & 0 & 0 \\
 9 & 12 & 9 & 4 & 1 & 0 & 0 \\
 21 & 30 & 25 & 14 & 5 & 1 & 0 \\
 51 & 76 & 69 & 44 & 20 & 6 & 1 \\
\end{array}
\right).$$ The production matrix takes the simple form
$$\left(
\begin{array}{ccccccc}
 1 & 1 & 0 & 0 & 0 & 0 & 0 \\
 1 & 1 & 1 & 0 & 0 & 0 & 0 \\
 0 & 1 & 1 & 1 & 0 & 0 & 0 \\
 0 & 0 & 1 & 1 & 1 & 0 & 0 \\
 0 & 0 & 0 & 1 & 1 & 1 & 0 \\
 0 & 0 & 0 & 0 & 1 & 1 & 1 \\
 0 & 0 & 0 & 0 & 0 & 1 & 1 \\
\end{array}
\right),$$ and hence the $Z$-sequence is the sequence $1,1,0,0,0,\ldots$ and the $A$-sequence is $1,1,1,0,0,0,\ldots$. Then $Z(x)=1+x$ and $A(x)=1+x+x^2$.

This Riordan array is also defined by the $A$-matrix
$$A=\left(
\begin{array}{ccc}
 1 & 0 & 1 \\
 1 & 1 & 1 \\
\end{array}
\right),$$ and the $\rho$ sequence $0,0,0,\ldots$.

In this example, the $A$-sequence is of a simple form, as is the $A$-matrix. In many cases, however, the $A$-matrix may be of a simple form, while the $A$-sequence is of a more complex nature.
\end{example}

\begin{example}
We now take the example where
$$A=\left(
\begin{array}{ccc}
 1 & 0 & 1 \\
 1 & 1 & 0 \\
\end{array}
\right),$$
and $\rho_n=0$ for all $n \ge 0$.
We find that $u=f(x)$ satisfies the equation
$$\frac{u}{x}=1+u^2+x(1+u),$$ or
$$f(x)=\frac{1-x^2-\sqrt{1-6x^2-4x^3+x^4}}{2x}=\frac{x}{1-x}c\left(\frac{x^2(1+x)}{(1-x^2)^2}\right).$$
The resulting Riordan array is then given by the Bell subgroup element $\left(\frac{f(x)}{x},f(x)\right)$. This begins
$$\left(
\begin{array}{ccccccc}
 1 & 0 & 0 & 0 & 0 & 0 & 0 \\
 1 & 1 & 0 & 0 & 0 & 0 & 0 \\
 2 & 2 & 1 & 0 & 0 & 0 & 0 \\
 3 & 5 & 3 & 1 & 0 & 0 & 0 \\
 7 & 10 & 9 & 4 & 1 & 0 & 0 \\
 13 & 24 & 22 & 14 & 5 & 1 & 0 \\
 31 & 52 & 57 & 40 & 20 & 6 & 1 \\
\end{array}
\right).$$
We have
$$\bar{f}(x)=\frac{\sqrt{1+4x+6x^2+x^4}-x^2-1}{2(1+x)},$$ and so $A(x)$ is given by
$$A(x)=\frac{1+x^2+\sqrt{1+4x+6x^2+x^4}}{2}.$$
Thus $a_n$ is the sequence that begins
$$1, 1, 1, -1, 2, -3, 3, 1, -15, 47, -98,\ldots.$$
The expansion of $f(x)$ in this case begins
$$ 0, 1, 1, 2, 3, 7, 13, 31, 65, 156, 351, 849,\ldots.$$
The expansion of $\frac{f(x)}{x}$ is given by \seqnum{A171416}. The Hankel transform of this sequence
begins
$$1, 1, 2, 3, 7, 23, 59, 314, 1529, 8209, 83313,\ldots$$ or \seqnum{A006720}$(n+2)$. This is essentially the Somos $4$ sequence.

In this example, we see that while the $A$-sequence is reasonably complicated, the $A$-matrix is simple. In addition, this $A$-matrix leads to a power series $f(x)$ such that the Hankel transform of the expansion of $\frac{f(x)}{x}$ is the Somos $4$ sequence.

In the sequel, we shall be interested in determining a form of $A$-matrix that will always lead to a Somos $4$ type sequence.
\end{example}

\section{General Somos $4$ sequences}
We call a sequence $s_n$ a $(\alpha, \beta)$ Somos $4$ sequence if the terms $s_n$ satisfy
$$ s_n=\frac{\alpha s_{n-1} s_{n-3} + \beta s_{n-2}^2}{s_{n-4}}, \quad n \ge 4.$$
In exceptional cases, where a divide by zero may occur, the more general condition
$$ s_n s_{n-4} = \alpha s_{n-1} s_{n-3} + \beta s_{n-2}^2$$ is appropriate. There are many proven and conjectured results concerning $(\alpha, \beta)$ Somos $4$ sequences and Riordan arrays \cite{PSI, Invariant}. These sequences are closely related to elliptic curves.

\section{$A$-matrices and Somos conjectures}
We consider the case
$$A=\left(
\begin{array}{ccc}
 1 & a & b \\
 1 & c & d \\
\end{array}
\right),$$ with $\rho_n=0$ for all $n$.
To find $u=f(x)$, we solve the equation
$$\frac{u}{x}=1+ au + bu^2+ x(1+cu+ du^2).$$
This gives us
$$f(x)=\frac{1-ax-cx^2-\sqrt{1-2ax-(a^2-2(2b+c))x^2+2(ac-2(b+d))x^3+(c^2-4d)x^4}}{2x(dx+b)}.$$
It follows that
$$\frac{f(x)}{x}=\frac{1+x}{1-ax-cx^2}c\left(\frac{x^2(1+x)(b+dx)}{(1-ax-cx^2)^2}\right).$$
This now allows us to make the following conjecture.
\begin{conjecture} The Hankel transform of the expansion of $\frac{f(x)}{x}$ is a
$((b + a b + d)^2, b^4 - b^3 (2 + 3 a + a^2 - 2 c) +  b (a + a^2 - a c - 2 d) d + (1 + a - c) d^2 -  b^2 (c + a c - c^2 + 2 d + 3 a d))$ Somos $4$ sequence.
\end{conjecture}
\begin{example}
The matrix
$$A=\left(
\begin{array}{ccc}
 1 & -2 & -1 \\
 1 & 1 & 0 \\
\end{array}
\right),$$ leads to the Riordan array with
$$\frac{f(x)}{x}=\frac{1+x}{1+2x-x^2}c\left(-\frac{x^2(1+x)}{(1+2x-x^2)^2}\right).$$
This expands to give the sequence that begins
$$ 1, -1, 2, -3, 3, 1, -15, 47, -98, 133, -17,\ldots.$$
The Hankel transform of this sequence begins
$$1, 1, -2, -1, 3, -5, -7, -4, 23, 29, -59,\ldots.$$
This is a $(1,1)$ Somos $4$ sequence (compare with \seqnum{A006769}).
The expansion $v_n$ of $\frac{f(x)}{x}$ thus satisfies the convolution recurrence \cite{Rec}
$$v_n= -2 v_{n-1}- v_{n-2}- \sum_{i=0}^{n-4} v_{i+1} v_{n-i-3},$$ with
$v_1=-1$ and $v_2=2$.
\end{example}
\begin{proposition} The Bell matrix $(t_{n,k})$ which has
$$A=\left(
\begin{array}{ccc}
 1 & a & b \\
 1 & c & d \\
\end{array}
\right)$$ as its $A$-matrix satisfies
$$t_{n,k}=t_{n-1,k-1}+at_{n-1,k}+ b t_{n-1,k+1}+t_{n-2,k-1}+c t_{n-2,k}+d t_{n-2,k+1},$$
with $t_{n,k}=0$ if $k<0$ or $k>n$, $t_{0,0}=1$ and $t_{1,0}=a+1$.
This is the Riordan array
$$\left(\frac{1+x}{1-ax-cx^2}c\left(\frac{x^2(1+x)(b+dx)}{(1-ax-cx^2)^2}\right), \frac{x(1+x)}{1-as-cx^2}c\left(\frac{s^2(1+x)(b+dx)}{(1-ax-cx^2)^2}\right)\right).$$
\end{proposition}
\begin{proof} The recurrence follows from the form of the $A$-matrix. We have calculated $f(x)$ above. The Bell element corresponding to this is then $\left(\frac{f(x)}{x}, f(x)\right)$. The value for $t_{1,0}$ follows from this.
\end{proof}
We note that the Riordan array
$$\left(\frac{1-(a-1)x}{1-ax-cx^2}c\left(\frac{x^2(1+x)(b+dx)}{(1-ax-cx^2)^2}\right), \frac{x(1+x)}{1-ax-cx^2}c\left(\frac{x^2(1+x)(b+dx)}{(1-ax-cx^2)^2}\right)\right)$$ satisfies the same recurrence with $t_{1,0}=1$. Note that we shall always assume that $t_{n,k}=0$ if $k<0$ or $k>n$, $t_{0,0}=1$ in the following.

\section{Further examples with $\rho_n=0$}
In the following examples, we choose $\rho_n=0$ for all $n \ge 0$.
\begin{example} We let
$$A=\left(
\begin{array}{ccc}
 1 & 1 & 0 \\
 1 & 1 & 1 \\
\end{array}
\right).$$
Then the triangle $(t_{n,k})$ with
$$t_{n,k}=t_{n-1,k-1}+t_{n-1,k}+t_{n-2,k-1}+ t_{n-2,k}+ t_{n-2,k+1},$$ and
$t_{1,0}=2$ begins
$$\left(
\begin{array}{ccccccc}
 1 & 0 & 0 & 0 & 0 & 0 & 0 \\
 2 & 1 & 0 & 0 & 0 & 0 & 0 \\
 3 & 4 & 1 & 0 & 0 & 0 & 0 \\
 6 & 10 & 6 & 1 & 0 & 0 & 0 \\
 13 & 24 & 21 & 8 & 1 & 0 & 0 \\
 29 & 59 & 62 & 36 & 10 & 1 & 0 \\
 66 & 146 & 174 & 128 & 55 & 12 & 1 \\
\end{array}
\right).$$
This is the Bell matrix
$$\left(\frac{1+x}{1-x-x^2}c\left(\frac{x^3(1+x)}{(1-x-x^2)^2}\right), \frac{x(1+x)}{1-x-x^2}c\left(\frac{x^3(1+x)}{(1-x-x^2)^2}\right)\right).$$
The Hankel transform of the sequence $t_{n,0}$ begins
$$1, -1, -4, -3, 19, 67, \ldots,$$ which is the $(1,1)$ Somos $4$ sequence \seqnum{A178628}$(n+1)$.
In fact, the generating function for $t_{n,0}$ may be expressed as the continued fraction \cite{Wall}
$$\cfrac{1}{1-2x+
\cfrac{x^2}{1+2x+
\cfrac{4x^2}{1-\frac{11}{4}x-
\cfrac{\frac{3}{16}x^2}{1-\frac{43}{12}x-\cdots}}}},$$ where $1,4,-\frac{3}{16},\ldots$ are the $x$-coordinates of the multiples of $(0,0)$ on the elliptic curve
$$y^2-xy-y=x^3+x^3+x.$$
\end{example}
\begin{example}
We let
$$A=\left(
\begin{array}{ccc}
 1 & 0 & 1 \\
 1 & 0 & 1 \\
\end{array}
\right).$$
Then the triangle $(t_{n,k})$ with
$$t_{n,k}=t_{n-1,k-1}+t_{n-1,k+1}+t_{n-2,k-1}+ t_{n-2,k+1},$$ and
$t_{1,0}=1$ begins
$$\left(
\begin{array}{ccccccc}
 1 & 0 & 0 & 0 & 0 & 0 & 0 \\
 1 & 1 & 0 & 0 & 0 & 0 & 0 \\
 1 & 2 & 1 & 0 & 0 & 0 & 0 \\
 3 & 3 & 3 & 1 & 0 & 0 & 0 \\
 5 & 8 & 6 & 4 & 1 & 0 & 0 \\
 11 & 17 & 16 & 10 & 5 & 1 & 0 \\
 25 & 38 & 39 & 28 & 15 & 6 & 1 \\
\end{array}
\right).$$
This is the Bell matrix
$$((1+x)c(x^2(1+x+x^2)), x(1+x)c(x^2(1+x+x^2))).$$
The sequence $t_{n,0}$ begins
$$1, 1, 1, 3, 5, 11, 25, 55, 129, 303, 721, 1743,\ldots.$$
This is \seqnum{A104545}, which counts the number of Motzkin paths of length $n$ having no consecutive $(1,0)$ steps (Emeric Deutsch). Its Hankel transform begins
$$1, 0, -4, -16, -64, 0, 4096, 65536, 1048576, 0, -1073741824,\ldots.$$ This is \seqnum{A162547}, the Somos $4$ variant $$s_n = (4s_{n-1}s_{n-3} - 4s_{n-2})^2) / s_{n-4}, \quad \text{ if } n \neq 4k+1,$$ otherwise $ s_n = 0$, with $s_{-2} = s_{-1} = s_0 = 1$.
\end{example}
\begin{example}  We let
$$A=\left(
\begin{array}{ccc}
 1 & 0 & 1 \\
 1 & 4 & 1 \\
\end{array}
\right).$$
Then the triangle $(t_{n,k})$ with
$$t_{n,k}=t_{n-1,k-1}+t_{n-1,k+1}+t_{n-2,k-1}+4 t_{n-2,k}+ t_{n-2,k+1},$$ and
$t_{1,0}=1$ begins
$$\left(
\begin{array}{ccccccc}
 1 & 0 & 0 & 0 & 0 & 0 & 0 \\
 1 & 1 & 0 & 0 & 0 & 0 & 0 \\
 5 & 2 & 1 & 0 & 0 & 0 & 0 \\
 7 & 11 & 3 & 1 & 0 & 0 & 0 \\
 33 & 24 & 18 & 4 & 1 & 0 & 0 \\
 63 & 105 & 52 & 26 & 5 & 1 & 0 \\
 261 & 262 & 231 & 92 & 35 & 6 & 1 \\
\end{array}
\right).$$
The sequence $t_{n,0}$ which begins
$$1, 1, 5, 7, 33, 63, 261, 619, 2333, 6355, 22669,\ldots$$ has a Hankel transform that begins
$$1, 4, 28, 304, 14272, 676864,\ldots.$$
This is a $(4,12)$ Somos $4$ sequence. In general, the $A$-matrix
$$A=\left(
\begin{array}{ccc}
 1 & 0 & 1 \\
 1 & r & 1 \\
\end{array}
\right)$$
leads to a $(4, r^2-4)$ Somos $4$ sequence.
\end{example}
\section{A quasi-involution}
We consider the case of the $A$-matrix 
$$A=\left(
\begin{array}{ccc}
 1 & 0 & 1 \\
 0 & 1 & 0 \\
\end{array}
\right).$$ We solve the equation
$$\frac{u}{x}=1+u^2+xu,$$ to obtain
$$u=f(x)=\frac{1-x^2-\sqrt{1-6x^2+x^4}}{2x}.$$ Thus the Bell matrix $(t_{n,k})$ defined by the above $A$-matrix is the Riordan array
$$\left(\frac{1-x^2-\sqrt{1-6x^2+x^4}}{2x^2},\frac{1-x^2-\sqrt{1-6x^2+x^4}}{2x}\right).$$
This triangle begins
$$\left(
\begin{array}{cccccccc}
 1 & 0 & 0 & 0 & 0 & 0 & 0 & 0 \\
 0 & 1 & 0 & 0 & 0 & 0 & 0 & 0 \\
 2 & 0 & 1 & 0 & 0 & 0 & 0 & 0 \\
 0 & 4 & 0 & 1 & 0 & 0 & 0 & 0 \\
 6 & 0 & 6 & 0 & 1 & 0 & 0 & 0 \\
 0 & 16 & 0 & 8 & 0 & 1 & 0 & 0 \\
 22 & 0 & 30 & 0 & 10 & 0 & 1 & 0 \\
 0 & 68 & 0 & 48 & 0 & 12 & 0 & 1 \\
\end{array}
\right).$$
The inverse of this matrix begins
$$\left(
\begin{array}{cccccccc}
 1 & 0 & 0 & 0 & 0 & 0 & 0 & 0 \\
 0 & 1 & 0 & 0 & 0 & 0 & 0 & 0 \\
 -2 & 0 & 1 & 0 & 0 & 0 & 0 & 0 \\
 0 & -4 & 0 & 1 & 0 & 0 & 0 & 0 \\
 6 & 0 & -6 & 0 & 1 & 0 & 0 & 0 \\
 0 & 16 & 0 & -8 & 0 & 1 & 0 & 0 \\
 -22 & 0 & 30 & 0 & -10 & 0 & 1 & 0 \\
 0 & -68 & 0 & 48 & 0 & -12 & 0 & 1 \\
\end{array}
\right).$$
We recall that a Riordan array $(g(x^2), x g(x^2))$ is called a quasi-involution if we have
$$(g(x^2), x g(x^2))^{-1}=(g(-x^2), x g(-x^2)).$$ This is the case for this example \cite{Structural}.
\section{The case of $r_n=0^n$}
In this section, we take the case of
$$A=\left(
\begin{array}{ccc}
 1 & a & b \\
 1 & c & d \\
\end{array}
\right)$$ and $\rho_0=1$, and $\rho_n=0$ for $n >0$.
Thus we must solve the equation
$$\frac{u}{x}=1+ a u + bu^2 + x(1+cu + du^2)+\rho_0\frac{u^2}{x}$$ to get
$f(x)=u$. We find that
$$\frac{f(x)}{x}=\frac{1+x}{1-ax-cx^2}c\left(\frac{x(1+x)(\rho_0+bx+dx^2)}{(1-ax-cx^2)^2}\right).$$
In this case, $\frac{f(x)}{x}$ expands to begin
$$1, a + \rho_0 + 1, a^2 + a(3\rho_0 + 1) + b + c + 2\rho_0^2 + 2\rho_0,\ldots,$$ and the corresponding Bell triangle begins $$\left(
\begin{array}{ccc}
 1 & 0 & 0 \\
 a+\rho_0+1 & 1 & 0 \\
 a^2+a (3 \rho_0+1)+b+c+2 \rho_0 (\rho_0+1) & 2 a+2 \rho_0+2 & 1 \\
\end{array}
\right).$$
In the case that $\rho_0=1$, we obtain
$$\frac{f(x)}{x}=\frac{1+x}{1-ax-cx^2}c\left(\frac{x(1+x)(1+bx+dx^2)}{(1-ax-cx^2)^2}\right).$$
\begin{example}
We consider the case of
$$A=\left(
\begin{array}{ccc}
 1 & 0 & 1 \\
 1 & 1 & 0 \\
\end{array}
\right)$$
and $\rho_n=0^n$. The resulting Bell triangle begins
$$\left(
\begin{array}{ccccccc}
 1 & 0 & 0 & 0 & 0 & 0 & 0 \\
 2 & 1 & 0 & 0 & 0 & 0 & 0 \\
 6 & 4 & 1 & 0 & 0 & 0 & 0 \\
 22 & 16 & 6 & 1 & 0 & 0 & 0 \\
 90 & 68 & 30 & 8 & 1 & 0 & 0 \\
 394 & 304 & 146 & 48 & 10 & 1 & 0 \\
 1806 & 1412 & 714 & 264 & 70 & 12 & 1 \\
\end{array}
\right).$$
This is the Bell matrix defined by the large Schroeder numbers \seqnum{A006318},
 $$(t_{n,k})=\left(\frac{1-x-\sqrt{1-6x+x^2}}{2x}, \frac{1-x-\sqrt{1-6x+x^2}}{2}\right),$$ and its $A$-sequence is the sequence
$$1,2,2,2,2,2,2,2,\ldots.$$
This follows since the solution $u(x)$ with $u(0)=0$ of the equation
$$\frac{u}{x}=1+u^2+x(1+u)+\frac{u^2}{x}$$ is given by
$$u(x)=\frac{1-x-\sqrt{1-6x+x^2}}{2}.$$
This is also the solution to the equation
$$\frac{u}{x}=1+u+\frac{u^2}{x}.$$
Thus a simpler $A$-matrix in this case is given by
$$A=\left(
\begin{array}{cc}
 1 & 1  \\
\end{array}
\right),$$ along with $\rho_n=0^n$.

In general, taking
$$A=\left(
\begin{array}{cc}
 1 & r  \\
\end{array}
\right),$$ along with $\rho_n=0^n$, leads to
$$\frac{f(x)}{x}=\frac{1-(r-1)x-\sqrt{1-2(r+1)x+(r-1)^2x^2}}{2x}=\frac{1}{1-(r-1)x}c\left(\frac{x}{(1-(r-1)x)^2}\right),$$ which is the generating function of the Narayana polynomials
$$1, r, r(r + 1), r(r^2 + 3r + 1), r(r^3 + 6r^2 + 6r + 1),\ldots.$$
In this case the triangle $(t_{n,k})$ begins
$$\left(
\begin{array}{ccccc}
 1 & 0 & 0 & 0 & 0 \\
 r & 1 & 0 & 0 & 0 \\
 r (r+1) & 2 r & 1 & 0 & 0 \\
 r \left(r^2+3 r+1\right) & r (3 r+2) & 3 r & 1 & 0 \\
 r \left(r^3+6 r^2+6 r+1\right) & 2 r \left(2 r^2+4 r+1\right) & 3 r (2 r+1) & 4 r & 1 \\
\end{array}
\right),$$ with $A$-sequence
$$1,r,r,r,\ldots.$$
\end{example}
\begin{example} We take
$$A=\left(
\begin{array}{ccc}
 1 & -2 & 2 \\
 1 & -1 & 1 \\
\end{array}
\right)$$ and $\rho_n=0^n$.
Then the Bell matrix that we obtain is \seqnum{A097609}, which begins
$$\left(
\begin{array}{ccccccc}
 1 & 0 & 0 & 0 & 0 & 0 & 0 \\
 0 & 1 & 0 & 0 & 0 & 0 & 0 \\
 1 & 0 & 1 & 0 & 0 & 0 & 0 \\
 1 & 2 & 0 & 1 & 0 & 0 & 0 \\
 3 & 2 & 3 & 0 & 1 & 0 & 0 \\
 6 & 7 & 3 & 4 & 0 & 1 & 0 \\
 15 & 14 & 12 & 4 & 5 & 0 & 1 \\
\end{array}
\right).$$
This is the Bell matrix
$$\left(\frac{1+x-\sqrt{1-2x-3x^2}}{2x(1-x)}, \frac{1+x-\sqrt{1-2x-3x^2}}{2(1-x)}\right)$$ defined by the so-called Motzkin sums \seqnum{A005043}. This triangle counts the number of Motzkin paths of length $n$ having $k$ horizontal steps at level $0$ (Emeric Deutsch). The $A$-sequence of this triangle is
$$1,0,1,1,1,1,\ldots.$$
\end{example}
We now turn to the issue of the Somos $4$ nature of the Hankel transform of the expansion of $\frac{f(x)}{x}$.
We have the following conjecture.
\begin{conjecture} The Hankel transform of the expansion of $\frac{f(x)}{x}$ in the case that
$$A=\left(
\begin{array}{ccc}
 1 & a & b \\
 1 & c & d \\
\end{array}
\right)$$
and $\rho_n=0^n$ is an $(\alpha, \beta)$ Somos $4$ sequence, where
$$\alpha=(4 + a^2 + 3b + a(4 + b) + c + d)^2,$$ and
\begin{align*}\beta&=-16 - a^5 + b^4 - 3a^4(3 + b) + 2b^3(-2 + c) - 8c - 4c^2 - c^3\\
 &+ b^2(-28 - c + c^2 - 8d) - 8d - 6cd - 2c^2d - d^2 - cd^2\\
 &- a^3(32 + 23b + 3b^2 + c + 2d) - 2b(20 + c^2 + 9d + d^2 + 3c(2 + d))\\
 &- a^2(56 + 18b^2 + b^3 + 10d + c(6 + d) + b(66 + c + 5d)) \\
 &- a(48 + 3b^3 + 2c^2 + 16d + d^2 + b^2(38 - 2c + 3d) + c(12 + 5d) + b(84 - c^2 + 19d + c(8 + d))).\end{align*}
 \end{conjecture}
\begin{example}
For
$$A=\left(
\begin{array}{ccc}
 1 & 0 & -1 \\
 1 & 0 & -2 \\
\end{array}
\right)$$ and $\rho_n=0^n$, we obtain the Bell triangle that begins
$$\left(
\begin{array}{ccccccc}
 1 & 0 & 0 & 0 & 0 & 0 & 0 \\
 2 & 1 & 0 & 0 & 0 & 0 & 0 \\
 3 & 4 & 1 & 0 & 0 & 0 & 0 \\
 4 & 10 & 6 & 1 & 0 & 0 & 0 \\
 2 & 20 & 21 & 8 & 1 & 0 & 0 \\
 -11 & 29 & 56 & 36 & 10 & 1 & 0 \\
 -59 & 10 & 117 & 120 & 55 & 12 & 1 \\
\end{array}
\right).$$ For this, we have
$$\frac{f(x)}{x}=(1+x)c(x(1+x)(1-x-2x^2)).$$
The Hankel transform of the expansion of this generating function begins
$$1,-1,3,4,5,-31,\ldots.$$
This is a $(1,1)$ Somos $4$ sequence.
\end{example}
\begin{example}
For
$$A=\left(
\begin{array}{ccc}
 1 & 0 & -1 \\
 1 & -2 & 2 \\
\end{array}
\right)$$ and $\rho_n=0^n$, we obtain the Bell matrix that begins
$$\left(
\begin{array}{ccccccc}
 1 & 0 & 0 & 0 & 0 & 0 & 0 \\
 2 & 1 & 0 & 0 & 0 & 0 & 0 \\
 1 & 4 & 1 & 0 & 0 & 0 & 0 \\
 0 & 6 & 6 & 1 & 0 & 0 & 0 \\
 4 & 4 & 15 & 8 & 1 & 0 & 0 \\
 17 & 9 & 20 & 28 & 10 & 1 & 0 \\
 41 & 50 & 27 & 56 & 45 & 12 & 1 \\
\end{array}
\right).$$
For this, we have
$$\frac{f(x)}{x}=\frac{1+x}{1+2x^2}c\left(\frac{x(1+x)(1-x+2x^2}{(1+2x^2)^2}\right).$$
The Hankel transform of the expansion of this generating function begins
$$1, -3,-13, 14, 465, 1819,\ldots.$$
This is a $(1,3)$ Somos $4$ sequence.
\end{example}

\section{The case $\rho(x)=1+rx$}
We briefly consider the case
$$A=\left(
\begin{array}{cc}
 1 & 1  \\
\end{array}
\right),$$ along with $\rho(x)=1+x$. We are thus led to the equation
$$\frac{u}{x}=1+u+\frac{u^2(1+u)}{x},$$ which has solution
$$u(x)=f(x)=\frac{1}{3}\left(\sqrt{4-3x}\sin\left(\frac{1}{3}\sin^{-1}\left(\frac{18x+11}{2(4-3x)^{\frac{3}{2}}}\right)\right)-1\right).$$
Equivalently, we have
$$f(x)= \text{Rev} \frac{x(1-x-x^2)}{1+x}.$$
We have that $\frac{f(x)}{x}$ expands to give the sequence that begins
$$1, 2, 7, 31, 154, 820, 4575, 26398, 156233,\ldots,$$ or \seqnum{A007863}, the number of ``hybrid binary trees'' with $n$ internal nodes. The corresponding Bell matrix begins
$$\left(
\begin{array}{ccccccc}
 1 & 0 & 0 & 0 & 0 & 0 & 0 \\
 2 & 1 & 0 & 0 & 0 & 0 & 0 \\
 7 & 4 & 1 & 0 & 0 & 0 & 0 \\
 31 & 18 & 6 & 1 & 0 & 0 & 0 \\
 154 & 90 & 33 & 8 & 1 & 0 & 0 \\
 820 & 481 & 185 & 52 & 10 & 1 & 0 \\
 4575 & 2690 & 1065 & 324 & 75 & 12 & 1 \\
\end{array}
\right).$$
The production matrix of this array starts
$$\left(
\begin{array}{cccccc}
 2 & 1 & 0 & 0 & 0 & 0 \\
 3 & 2 & 1 & 0 & 0 & 0 \\
 5 & 3 & 2 & 1 & 0 & 0 \\
 8 & 5 & 3 & 2 & 1 & 0 \\
 13 & 8 & 5 & 3 & 2 & 1 \\
 21 & 13 & 8 & 5 & 3 & 2 \\
\end{array}
\right),$$
indicating that
$$A(x)=\frac{1+x}{1-x-x^2}, \quad\text{ and } Z(x)=\frac{2+x}{1-x-x^2}.$$
Here, we recognize the shifted Fibonacci numbers \seqnum{A000045}. The Bell matrix $(t_{n,k})$ in this case is given by
$$\left(\frac{1-x-x^2}{1+x}, \frac{x(1-x-x^2)}{1+x}\right)^{-1}.$$
Similarly, the $A$-matrix
$$A=\left(
\begin{array}{cc}
 1 & 2  \\
\end{array}
\right),$$ along with $\rho(x)=1+x$, leads to the Bell matrix
$$\left(\frac{1-x-x^2}{1+2x}, \frac{x(1-x-x^2)}{1+2x}\right)^{-1}$$ which begins
$$\left(
\begin{array}{ccccccc}
 1 & 0 & 0 & 0 & 0 & 0 & 0 \\
 3 & 1 & 0 & 0 & 0 & 0 & 0 \\
 13 & 6 & 1 & 0 & 0 & 0 & 0 \\
 70 & 35 & 9 & 1 & 0 & 0 & 0 \\
 424 & 218 & 66 & 12 & 1 & 0 & 0 \\
 2756 & 1437 & 471 & 106 & 15 & 1 & 0 \\
 18778 & 9876 & 3390 & 856 & 155 & 18 & 1 \\
\end{array}
\right).$$
The case of
$$A=\left(
\begin{array}{cc}
 1 & 1  \\
\end{array}
\right),$$ along with $\rho(x)=1+2x$ is especially interesting.
We obtain the Bell matrix that begins
$$\left(
\begin{array}{ccccccc}
 1 & 0 & 0 & 0 & 0 & 0 & 0 \\
 2 & 1 & 0 & 0 & 0 & 0 & 0 \\
 8 & 4 & 1 & 0 & 0 & 0 & 0 \\
 40 & 20 & 6 & 1 & 0 & 0 & 0 \\
 224 & 112 & 36 & 8 & 1 & 0 & 0 \\
 1344 & 672 & 224 & 56 & 10 & 1 & 0 \\
 8448 & 4224 & 1440 & 384 & 80 & 12 & 1 \\
\end{array}
\right),$$ which has a production matrix that begins
$$\left(
\begin{array}{cccccc}
 2 & 1 & 0 & 0 & 0 & 0 \\
 4 & 2 & 1 & 0 & 0 & 0 \\
 8 & 4 & 2 & 1 & 0 & 0 \\
 16 & 8 & 4 & 2 & 1 & 0 \\
 32 & 16 & 8 & 4 & 2 & 1 \\
 64 & 32 & 16 & 8 & 4 & 2 \\
\end{array}
\right).$$
Thus we have
$$A(x)=\frac{1}{1-2x}, \quad\text{ and } Z(x)=\frac{2}{1-2x}.$$
The corresponding Bell matrix  $(t_{n,k})$ is given by
$$(1-2x, x(1-2x))^{-1}.$$
In effect, the solution to the equation
$$\frac{u}{x}=1+u+\frac{u^2(1+2u)}{x}$$ is given by
$$u(x)=f(x)=\frac{1-\sqrt{1-8x}}{4}.$$ This is the generating function of $2^n C_n$ \seqnum{A151374}.
Extending the two previous cases, we look at the case
$$A=\left(
\begin{array}{cc}
 1 & 2  \\
\end{array}
\right),$$ along with $\rho(x)=1+2x$.
The equation to be solved is now
$$\frac{u}{x}=1+2u+\frac{u^2(1+2u)}{x}$$ with solution
$$u(x)=f(x)=\frac{1}{3}\left(\sqrt{7-12x}\sin\left(\frac{1}{3}\sin^{-1}\left(\frac{2(18x+5)}{(7-12x)^{\frac{3}{2}}}\right)\right)-\frac{1}{2}\right).$$
Equivalently, we have
$$f(x)=\text{Rev} \frac{x(1-x-2x^2)}{1+2x}.$$
The Bell matrix $(t_{n,k})$ in this case begins
$$\left(
\begin{array}{ccccccc}
 1 & 0 & 0 & 0 & 0 & 0 & 0 \\
 3 & 1 & 0 & 0 & 0 & 0 & 0 \\
 14 & 6 & 1 & 0 & 0 & 0 & 0 \\
 83 & 37 & 9 & 1 & 0 & 0 & 0 \\
 554 & 250 & 69 & 12 & 1 & 0 & 0 \\
 3966 & 1802 & 528 & 110 & 15 & 1 & 0 \\
 29756 & 13580 & 4122 & 944 & 160 & 18 & 1 \\
\end{array}
\right),$$ with a production matrix that begins
$$\left(
\begin{array}{cccccc}
 3 & 1 & 0 & 0 & 0 & 0 \\
 5 & 3 & 1 & 0 & 0 & 0 \\
 11 & 5 & 3 & 1 & 0 & 0 \\
 21 & 11 & 5 & 3 & 1 & 0 \\
 43 & 21 & 11 & 5 & 3 & 1 \\
 85 & 43 & 21 & 11 & 5 & 3 \\
\end{array}
\right).$$ We recognise the shifted Jacobsthal numbers \seqnum{A001045}. We deduce that the above Bell matrix is given by
$$\left(\frac{1-x-2x^2}{1+2x}, \frac{x(1-x-2x^2)}{1+2x}\right)^{-1}.$$ The sequence $t_{n,0}$ is given by \seqnum{A215661}.

In general, the Bell matrix $(t_{n,k})$ defined by
$$A=\left(
\begin{array}{cc}
 1 & r  \\
\end{array}
\right),$$ along with $\rho(x)=1+rx$ is given by
$$\left(\frac{1-x-rx^2}{1+rx}, \frac{x(1-x-rx^2)}{1+rx}\right)^{-1}.$$
We have
$$t_{n,k}=t_{n-1,k-1}+r t_{n-1,k}+t_{n,k+1}+r t_{n,k+2},$$ with
$t_{0,0}=1$, $t_{1,0}=r+1$.

The sequence $t_{n,0}$ (the first column of the Bell matrix) then begins
$$1, r + 1, r^2 + 4r + 2, r^3 + 10r^2 + 15r + 5, r^4 + 20r^3 + 63r^2 + 56r + 14,\ldots.$$
This polynomial sequence $P_n(r)$ can be expressed as
$$P_n(r)=\sum_{k=0^n} \frac{1}{k+(2n+1)0^k}\binom{2n-k}{k-1+0^k} \binom{2n-k+1}{n-k}r^k,$$ where the coefficient array begins
$$\left(
\begin{array}{ccccccc}
 1 & 0 & 0 & 0 & 0 & 0 & 0 \\
 1 & 1 & 0 & 0 & 0 & 0 & 0 \\
 2 & 4 & 1 & 0 & 0 & 0 & 0 \\
 5 & 15 & 10 & 1 & 0 & 0 & 0 \\
 14 & 56 & 63 & 20 & 1 & 0 & 0 \\
 42 & 210 & 336 & 196 & 35 & 1 & 0 \\
 132 & 792 & 1650 & 1440 & 504 & 56 & 1 \\
\end{array}
\right).$$
An interesting feature of this array is the following. The diagonal sums of this array begin
$$1, 1, 3, 9, 30, 108, 406, 1577, 6280, 25499, 105169,\ldots.$$
This is \seqnum{A200074}. The generating function $g(x)$ of this sequence can be expressed as $g(x)=\frac{u(x)}{x}$, where we have
$$\frac{u}{x}=1+u^3+xu+\frac{u^2}{x}.$$ 
Thus this sequence is the first column of the Bell matrix $(t_{n,k})$ defined by the $A$-matrix
$$\left(
\begin{array}{cccc}
 1 & 0 & 0 & 1 \\
 0 & 1 & 0 & 0 \\
\end{array}
\right),$$ with $\rho_n=0^n$. This triangle $(t_{n,k})$ begins
$$\left(
\begin{array}{ccccccc}
 1 & 0 & 0 & 0 & 0 & 0 & 0 \\
 1 & 1 & 0 & 0 & 0 & 0 & 0 \\
 3 & 2 & 1 & 0 & 0 & 0 & 0 \\
 9 & 7 & 3 & 1 & 0 & 0 & 0 \\
 30 & 24 & 12 & 4 & 1 & 0 & 0 \\
 108 & 87 & 46 & 18 & 5 & 1 & 0 \\
 406 & 330 & 180 & 76 & 25 & 6 & 1 \\
\end{array}
\right).$$
We have
$$t_{n,k}=t_{n-1,k-1}+t_{n-1,k+2}+t_{n-2,k}+t_{n,k+1},$$ with
$t_{1,0}=1$.

\section{Perturbed orthogonal polynomials}
We consider the case when
$$A=\left(
\begin{array}{ccc}
 1 & a & b 
\end{array}
\right),$$ 
 and $\rho_n=c 0^n$.
 
Thus we must solve the equation
$$\frac{u}{x}=1+ a u + b u^2 + \frac{cu^2}{x}.$$ The required solution is
$$u = \frac{1-ax-\sqrt{1-2(a+2c)x+(a^2-4b)x^2}}{2(bx+c)}=\frac{x}{1-ax}c\left(\frac{x(bx+c)}{(1-ax)^2}\right).$$
Alternatively, we have
$$u=\text{Rev} \frac{x(1-cx)}{1+ax+bx^2}.$$
Thus the Bell matrix $(t_{n,k})$ with
$$t_{n,k}=t_{n-1,k-1}+ a t_{n-1,k}+ b t_{n-1,k+1}+ c t_{n,k+1}$$ and $t_{1,0}=a+c$ is given by
$$\left(\frac{1-cx}{1+a x + b x^2}, \frac{x(1-cx)}{1+ax+bx^2}\right)^{-1},$$
or
$$\left(\frac{1}{x}\text{Rev}\frac{x(1-cx)}{1+ax+bx^2},\text{Rev}\frac{x(1-cx)}{1+ax+bx^2}\right).$$
In the case that $c=0$, this is the moment matrix of the orthogonal polynomials
$$P_n(x)=(x-a)P_{n-1}(x)-b P_{n-2}(x),$$ with $P_0(x)=1$, $P_1(x)=x-a$.

Similarly, the triangle $(t_{n,k})$ with
$$t_{n,k}=t_{n-1,k-1}+ a t_{n-1,k}+ b t_{n-1,k+1}+c t_{n-1,k+2}+ d t_{n,k+1}$$ and $t_{1,0}=a+d$ is given by
$$\left(\frac{1}{x}\text{Rev}\frac{x(1-dx)}{1+ax+bx^2+cx^3},\text{Rev}\frac{x(1-dx)}{1+ax+bx^2+cx^3}\right).$$

The binomial transform of the Riordan array
$$\left(\frac{1}{1+a x + b x^2}, \frac{x}{1+ax+bx^2}\right)^{-1}$$ is the Riordan array
$$\left(\frac{1}{1+(a+1) x + b x^2}, \frac{x}{1+(a+1)x+bx^2}\right)^{-1}.$$ In a similar vein, we have the following result.
\begin{proposition} The binomial transform of $\frac{u}{x}$ where
$$\frac{u}{x}=1 + a u+ b u^2 + \frac{c u^2}{x}$$ is given by $\frac{v}{x}$ where
$$\frac{v}{x}=\frac{1}{1-x} (1+a v + b v^2) + \frac{c v^2}{x}.$$
\end{proposition}
\begin{proof}
We have $$\frac{u}{x}=\frac{1}{1-ax}c\left(\frac{x(bx+c)}{(1-ax)^2}\right).$$
We find that the binomial transform $\frac{1}{1-x} u\left(\frac{x}{1-x}\right)$ of $u$ is given by
$$\frac{1}{1-x} u\left(\frac{x}{1-x}\right)=\frac{1}{1-(a+1)x}c\left(\frac{x(x(b-c)+c)}{(1-(a+1)x)^2}\right).$$
But this is precisely $\frac{v}{x}$ where
$$\frac{v}{x}=\frac{1}{1-x} (1+a v + b v^2) + \frac{c v^2}{x}.$$
\end{proof}
Thus the Riordan array with $A$-matrix $A=(1,a,b)$ and $\rho_n=c 0^n$ will have a binomial transform with the infinite $A$-matrix
$$\left(
\begin{array}{ccc}
 1 & a & b \\
 1 & a & b \\
 1 & a & b \\
 1 & a & b \\
 \vdots & \vdots & \vdots \\
\end{array}
\right)$$ and $\rho_n=c 0^n$.
\begin{example}
The Riordan array
$$\left(\frac{1-5x}{1+2x+3x^2}, \frac{x(1-5x)}{1+2x+3x^2}\right)^{-1}$$ begins
$$\left(
\begin{array}{ccccccc}
 1 & 0 & 0 & 0 & 0 & 0 & 0 \\
 7 & 1 & 0 & 0 & 0 & 0 & 0 \\
 87 & 14 & 1 & 0 & 0 & 0 & 0 \\
 1331 & 223 & 21 & 1 & 0 & 0 & 0 \\
 22731 & 3880 & 408 & 28 & 1 & 0 & 0 \\
 415427 & 71665 & 7990 & 642 & 35 & 1 & 0 \\
 7949259 & 1380682 & 159591 & 14004 & 925 & 42 & 1 \\
\end{array}
\right).$$
Note for instance that
$$3880=1331 + 2 \cdot 223 + 3\cdot 21 + 5 \cdot 408.$$
Consider the binomial transform of this matrix, given by
$$\left(
\begin{array}{ccccccc}
 1 & 0 & 0 & 0 & 0 & 0 & 0 \\
 8 & 1 & 0 & 0 & 0 & 0 & 0 \\
 102 & 16 & 1 & 0 & 0 & 0 & 0 \\
 1614 & 268 & 24 & 1 & 0 & 0 & 0 \\
 28606 & 4860 & 498 & 32 & 1 & 0 & 0 \\
 543298 & 93440 & 10250 & 792 & 40 & 1 & 0 \\
 10810754 & 1873548 & 214086 & 18296 & 1150 & 48 & 1 \\
\end{array}
\right).$$
Note now that, for instance, we have
$$10250=1 \cdot(4860 + 268 + 16 + 1) + 2\cdot(498 + 24 + 1) + 3\cdot(32 + 1) + 5 \cdot 792.$$
\end{example}

\section{Calculating the $A$-sequence}
We give an example of how to calculate the $A$ sequence, given a suitable $A$-matrix and $\rho$ sequence.
Thus we take
$$A=\left(
\begin{array}{ccc}
 1 & 1 & 1 \\
 0 & 1 & 0 \\
\end{array}\right)$$ and $\rho_n=0^n$.
This leads us to the equation
$$\frac{u}{x}=1+u+u^2+xu+\frac{u^2}{x},$$
whose solution is given by
$$u=f(x)=\frac{1-x-x^2-\sqrt{1-6x-5x^2+2x^3+x^4}}{2(1+x)}=\frac{x}{1-x-x^2}c\left(\frac{x(1+x)}{(1-x-x^2)^2}\right).$$ From this we can calculate the $A$-sequence since its generating function $A(x)$ is given by
$$A(x)=\frac{x}{\bar{f}(x)}.$$
An alternative method is to start with the equation
$$\frac{u}{x}=1+u+u^2+xu+\frac{u^2}{x},$$ which we write using $u=f(x)$ as
$$\frac{f(x)}{x}=1+f(x)+f(x)^2+xf(x)+\frac{(f(x))^2}{x}.$$
Substituting $\bar{f}(x)$ for $x$, and using the fact that $f(\bar{f}(x))=x$, we then obtain
$$\frac{x}{\bar{f}(x)}=1+x+x^2+x \bar{f}(x)+\frac{x^2}{\bar{f}(x)}.$$
Thus we must solve the equation
$$v=1+x+x^2+\frac{x^2}{v}+xv,$$ for $v=A(x)=\frac{x}{\bar{f}(x)}$.
We obtain that
$$A(x)=v=\frac{1+x+x^2+\sqrt{1+2x+7x^2-2x^3+x^4}}{2(1-x)}.$$
This expands to give the sequence that begins
$$1, 2, 4, 2, 2, 8, -2, -10, 52, -26, -202, 576, \ldots.$$
The Bell matrix $(t_{n,k})$ in this case begins
$$\left(
\begin{array}{ccccccccc}
 1 & 0 & 0 & 0 & 0 & 0 & 0 & 0 & 0 \\
 2 & 1 & 0 & 0 & 0 & 0 & 0 & 0 & 0 \\
 8 & 4 & 1 & 0 & 0 & 0 & 0 & 0 & 0 \\
 34 & 20 & 6 & 1 & 0 & 0 & 0 & 0 & 0 \\
 162 & 100 & 36 & 8 & 1 & 0 & 0 & 0 & 0 \\
 820 & 524 & 206 & 56 & 10 & 1 & 0 & 0 & 0 \\
 4338 & 2832 & 1182 & 360 & 80 & 12 & 1 & 0 & 0 \\
 23694 & 15704 & 6828 & 2248 & 570 & 108 & 14 & 1 & 0 \\
 132612 & 88876 & 39818 & 13856 & 3850 & 844 & 140 & 16 & 1 \\
\end{array}
\right),$$ and its production matrix takes the form
$$\left(
\begin{array}{cccccccc}
 2 & 1 & 0 & 0 & 0 & 0 & 0 & 0 \\
 4 & 2 & 1 & 0 & 0 & 0 & 0 & 0 \\
 2 & 4 & 2 & 1 & 0 & 0 & 0 & 0 \\
 2 & 2 & 4 & 2 & 1 & 0 & 0 & 0 \\
 8 & 2 & 2 & 4 & 2 & 1 & 0 & 0 \\
 -2 & 8 & 2 & 2 & 4 & 2 & 1 & 0 \\
 -10 & -2 & 8 & 2 & 2 & 4 & 2 & 1 \\
 52 & -10 & -2 & 8 & 2 & 2 & 4 & 2 \\
\end{array}
\right).$$

\section{Conclusions}
We have given examples of Riordan arrays defined by simple $A$-matrices and simple $\   rho$ sequences, and we have shown the form of the calculations required to go from these data to a corresponding Bell matrix. In some cases, we have examined the associated $A$-sequence. Some special matrices, including a Riordan quasi-involution and certain ``perturbed'' moment matrices have been studied in terms of their defining $A$-matrix and $\rho$ sequence.

We have conjectured that an $A$-matrix of the form
$$A=\left(
\begin{array}{ccc}
 1 & a & b \\
 1 & c & d \\
\end{array}
\right)$$ leads to Hankel transforms that are $(\alpha, \beta)$ Somos $4$ sequences, in the case that $\rho_n=0$ for all $n$ and when $\rho_n=0^n$.  Specific examples bear this out, but the form of $\beta$ in general makes it uncertain at the moment how such a conjecture might be proven, or generalized. Other open questions that remain are the relationship between the parameters $(a,b,c,d)$ and the corresponding coefficients of the related elliptic curves. It is clear that these problems deserve further study.

\bigskip
\hrule
\bigskip
\noindent 2010 {\it Mathematics Subject Classification}: Primary 15B36;
 Secondary 05A15, 11C20, 11B37, 11B83, 15B36.
\noindent \emph{Keywords:}  Number triangle, Riordan array, $A$-matrix, $A$-sequence, Somos $4$ sequence, Hankel transform.

\bigskip
\hrule
\bigskip
\noindent (Concerned with sequences
\seqnum{A000045},
\seqnum{A000108},
\seqnum{A001045},
\seqnum{A005043},
\seqnum{A006318},
\seqnum{A006720},
\seqnum{A006769},
\seqnum{A007863},
\seqnum{A097609},
\seqnum{A104545},
\seqnum{A151374},
\seqnum{A162547},
\seqnum{A171416},
\seqnum{A178628}, and
\seqnum{A215661}.)

\end{document}